\documentclass[preprint,12pt,3p]{elsarticle}
\usepackage{amsmath,amsthm,amsfonts,amssymb}
\usepackage{hyperref}
\usepackage{yhmath}
\usepackage{MnSymbol}  
\usepackage[normalem]{ulem}
\usepackage{sectsty}
\usepackage{float}
\usepackage[caption = false]{subfig}
\usepackage{colortbl}
\usepackage[dvipsnames]{xcolor}
\newtheorem{theorem}{Theorem}[section]
\newtheorem{lemma}[theorem]{Lemma}

\newtheorem{remark}[theorem]{Remark}
\newtheorem{thm}{Theorem}[section]
\newtheorem{corollary}{Corollary}[section]

\newtheorem{ex}[thm]{Example}

\newtheorem{rem}[thm]{Remark}

\numberwithin{equation}{section}

\definecolor{darkslategray}{rgb}{0.18, 0.31, 0.31}
\definecolor{warmblack}{rgb}{0.0, 0.26, 0.26}
\definecolor{astral}{RGB}{46,116,181}
\subsectionfont{\color{astral}}
\sectionfont{\color{astral}}

\begin{document}
	\begin{frontmatter}
		 \title{ \textcolor{warmblack}{\bf 
  Generalized numerical radius  inequalities for certain operator matrices}}

		\author[label1]{Satyajit Sahoo}\ead{ssahoomath@gmail.com }

		\address[label1]{Department of Mathematics, 
School of Basic Sciences, Indian Institute of Technology Bhubaneswar, Odisha 752050, India.}
		
	\author[label2]{Narayan Behera}\ead{narayanbehera2004@gmail.com}

		\address[label2]{Space and Atmospheric Sciences Division, Physical Research Laboratory, Navrangpura, Ahmedabad 380009, India}


		\begin{abstract}
			\textcolor{warmblack}{
   		In this article, a series of new inequalities involving the $q$-numerical radius for $n\times n$ tridiagonal, and anti-tridiagonal operator matrices has been established. These inequalities serve to establish both lower and upper bounds for the $q$-numerical radius of operator matrices. Additionally, we developed $q$-numerical radius inequalities for $n\times n$ circulant, skew circulant, imaginary circulant, imaginary skew circulant operator matrices.
Important examples have been used to illustrate the developed inequalities. In this regard, analytical expressions and a numerical algorithm have also been employed to obtain the $q$-numerical radii.       
     We also provide a concluding section, which may lead to several new problems in this area.}
		\end{abstract}
		
	\end{frontmatter}

	\section{Introduction}\label{intro}
The $q$-numerical radius of a matrix is a generalization of the classical numerical radius, where a matrix is treated as an operator acting on a Hilbert space. It is an important quantity in operator theory and functional analysis, particularly in studying the spectrum of matrices and their geometric properties. In certain cases, the $q$-numerical radius has been studied to examine how it behaves with respect to different classes of matrices. The parameter 
$q$ can control the behavior of the matrix in various ways, depending on its structure.

	The operator matrices such as circulant, reverse circulant, symmetric circulant, $k$-circulant, Toeplitz matrices etc. \cite{JZ,Davis} play a crucial role in pure as well as applied mathematical researches such as graph theory, image processing, block filtering design, signal processing, regular polygon solutions, encoding, control and system theory, network, etc. 
	The norm estimation for the operator matrices \cite{BhatiaKit,Kissin} is extensively carried out in the past and it is widely used in operator theory, quantum information theory, mathematical physics, numerical analysis, etc. The norms of some circulant type matrices were determined by various mathematicians. For instance, Li et al. \cite{LIJL}, gave four kinds of norms for circulant and left circulant matrices involving special numbers. Bose et al. \cite{BoseHS}, discussed the convergence in probability and the convergence in distribution of the spectral norms of scaled Toeplitz, circulant, reverse (left) circulant, symmetric circulant, and $k$-circulant matrices. Works on norm equalities and inequalities of special kind of operator matrices can be found in the literature \cite{AUK,BhatiaKLi,KC,KC1}. Jiang and Xu \cite{JXU} explored special cases for norm equalities and inequalities, such as usual operator norm and Schatten $p$-norms. Several norm equalities and inequalities for the circulant, skew circulant, and $w$-circulant operator matrices were studied \cite{BaniKit,JXU}.

	\par 
	Let $\mathcal{H}$ be a complex Hilbert space with inner product $\langle \cdot,\cdot\rangle$  and the corresponding norm $\|\cdot\|$.
	Let  $\mathcal{B}(\mathcal{H})$ be the $C^*$-algebra of all bounded linear operators on $\mathcal{H}$. Let $\mathbb{H}=\displaystyle\bigoplus_{i=1}^n\mathcal{H}$ be the direct sum of $n$ copies of $\mathcal{H}$. If $S_{ij}, 1\leq i, j\leq n$ are operators in $\mathcal{B(H)}$, then operator matrix $\mathbb{S}=[S_{i,j}]$ can be defined on $\mathbb{H}$ by $ \mathbb{S}x=\begin{bmatrix}
	\displaystyle\sum_{j=1}^{n}S_{1j}x_j\\
	\vdots\\
	\displaystyle\sum_{j=1}^{n}S_{nj}x_j
	\end{bmatrix}$ for every vector $x=[x_1,\dots, x_n]^T\in \mathbb{H}$. If $S_i\in \mathcal{B(H)},  i=1,\dots,n$, then their direct sum, $\displaystyle\bigoplus_{i=1}^n S_i$, (which is an $n\times n$ block diagonal operator matrix), is given by	$ \displaystyle\bigoplus_{i=1}^nS_i =\begin{bmatrix}
	S_1 & &&\\
	& S_2&&\\
	&& \ddots\\
	&&& S_n
	\end{bmatrix}.$
If $S_{ij}$ are operators in $\mathcal{B(H)}, i=1,\dots,n$, then the general $n\times n$ tridiagonal operator matrix is defined as 	$\mathbb{S}_{tri}=\begin{bmatrix}
	S_{11} & S_{12}  & O & \cdots  &   O\\
	S_{21}  & S_{22} &  S_{23}&\cdots  & O\\
	O  & S_{32} &  S_{33}&\ddots  &  O\\
	\vdots & \vdots &\ddots & \ddots & S_{n-1 n} \\
	O  & O  &\cdots&  S_{n n-1} & S_{nn}
	\end{bmatrix}$. Also, the general $n\times n$ anti-tridiagonal operator matrix is defined as 
    
$\mathbb{S}_{atri}=\begin{bmatrix}
O &\cdots  &O & S_{1n-1}  &   S_{1n}\\
\vdots & \adots & S_{2n-2}&S_{2n-1}  &S_{2n}\\
	O  & \adots &  S_{3n-2}&S_{3n-1}  &  O\\
	S_{n-11}& \adots &\adots & \adots & \vdots \\
	S_{n1}  &S_{n2}  &O& \cdots & O
	\end{bmatrix}$.
    Typical tridiagonal and anti-tridiagonal matrix have been represented graphically in Figure \ref{Fig_Trid-0}.

\begin{figure}[H]
\centering
 {\includegraphics[scale=.80]{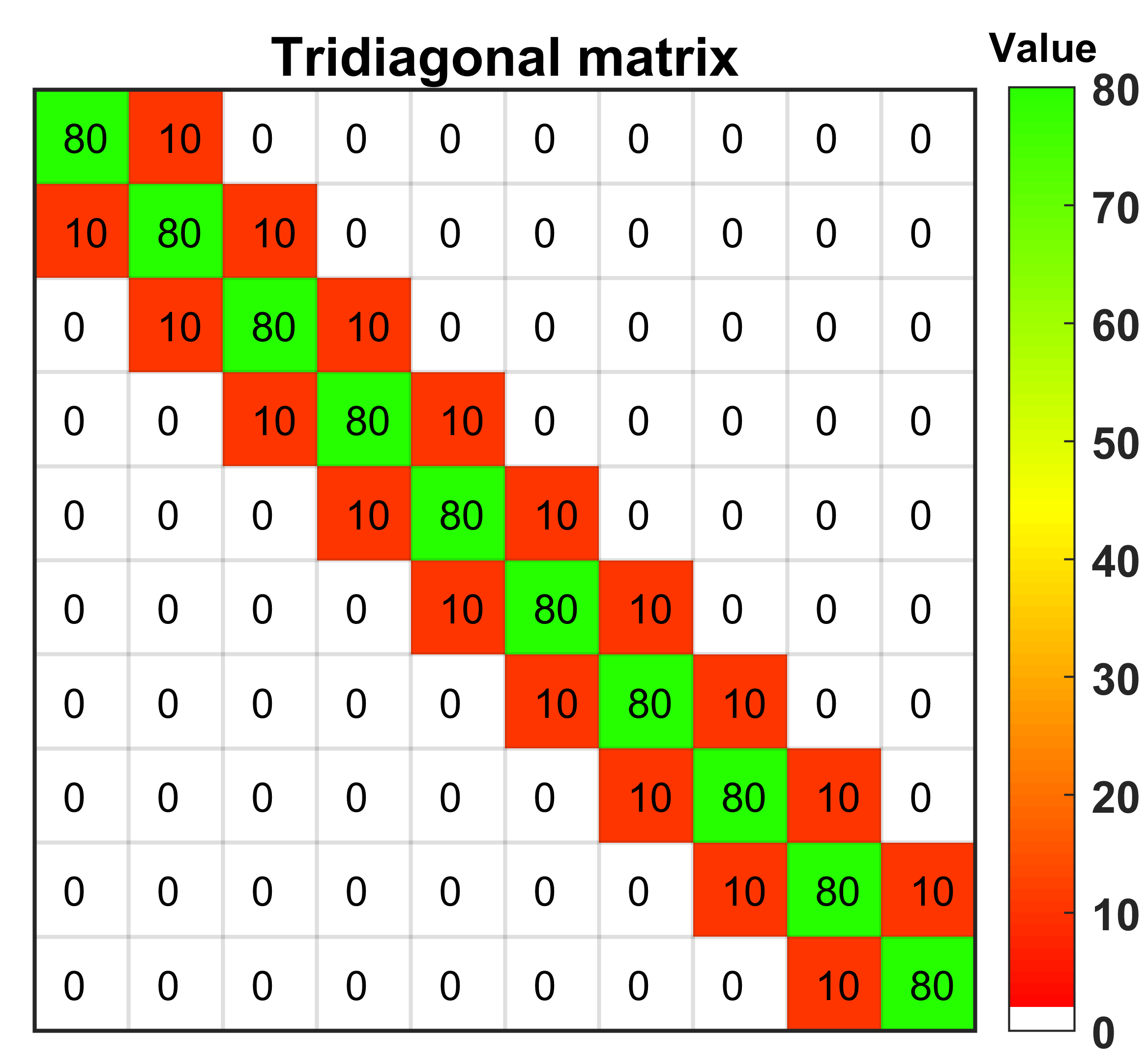}}
 \hspace{0.3cm}
 {\includegraphics[scale=.80]{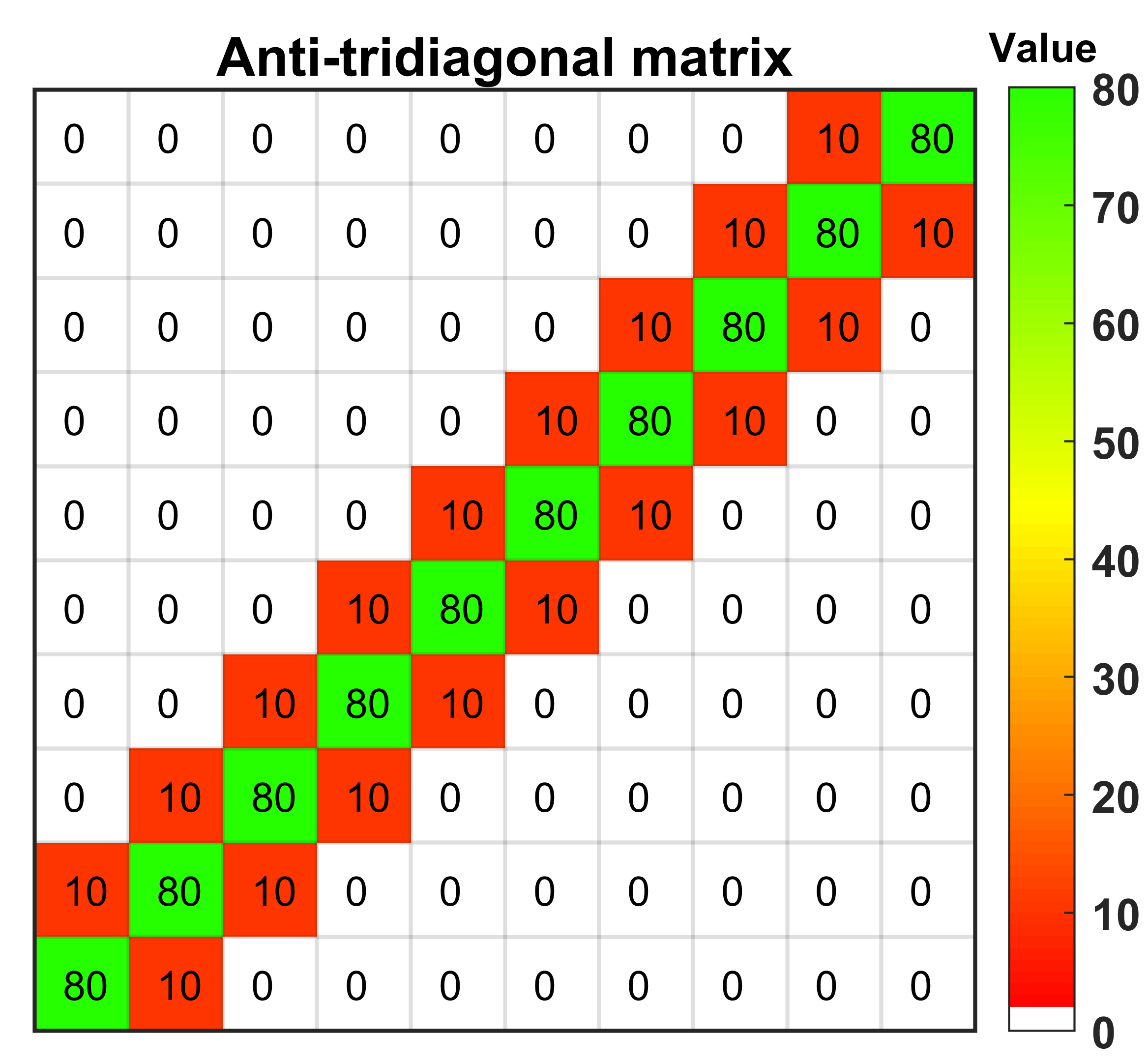}}
\caption{Graphical representation a tridiagonal and anti-tridiagonal matrix.}	 
\label{Fig_Trid-0}
\end{figure}
Similarly, if $S_i\in \mathcal{B(H)}, i=1,\dots,n$, then the circulant operator matrix $\mathbb{S}_{circ}=\mbox{circ}(S_1,\dots,S_n)$ is the $n\times n$ matrix whose first row has entries $S_1, \dots, S_n$ and the other rows are obtained by successive cyclic permutations of these entries, i.e.,
	$\mathbb{S}_{circ}=\begin{bmatrix}
	S_{1} & S_{2}  & S_{3} & \cdots  &   S_{n}\\
	S_{n}  & S_{1} &  S_{2}&\cdots  & S_{n-1}\\
	S_{n-1}  & S_{n} &  S_{1}&\ddots  &  S_{n-2}\\
	\vdots & \vdots &\ddots & \ddots & \vdots \\
	S_{2}  & S_{3}  &\cdots&  S_{n} & S_{1}
	\end{bmatrix}$. 
A typical circulant matrix has been represented graphically in Figure \ref{Fig_Circ-0}. 
 \begin{figure}[H]
\centering
{\includegraphics[scale=.80]{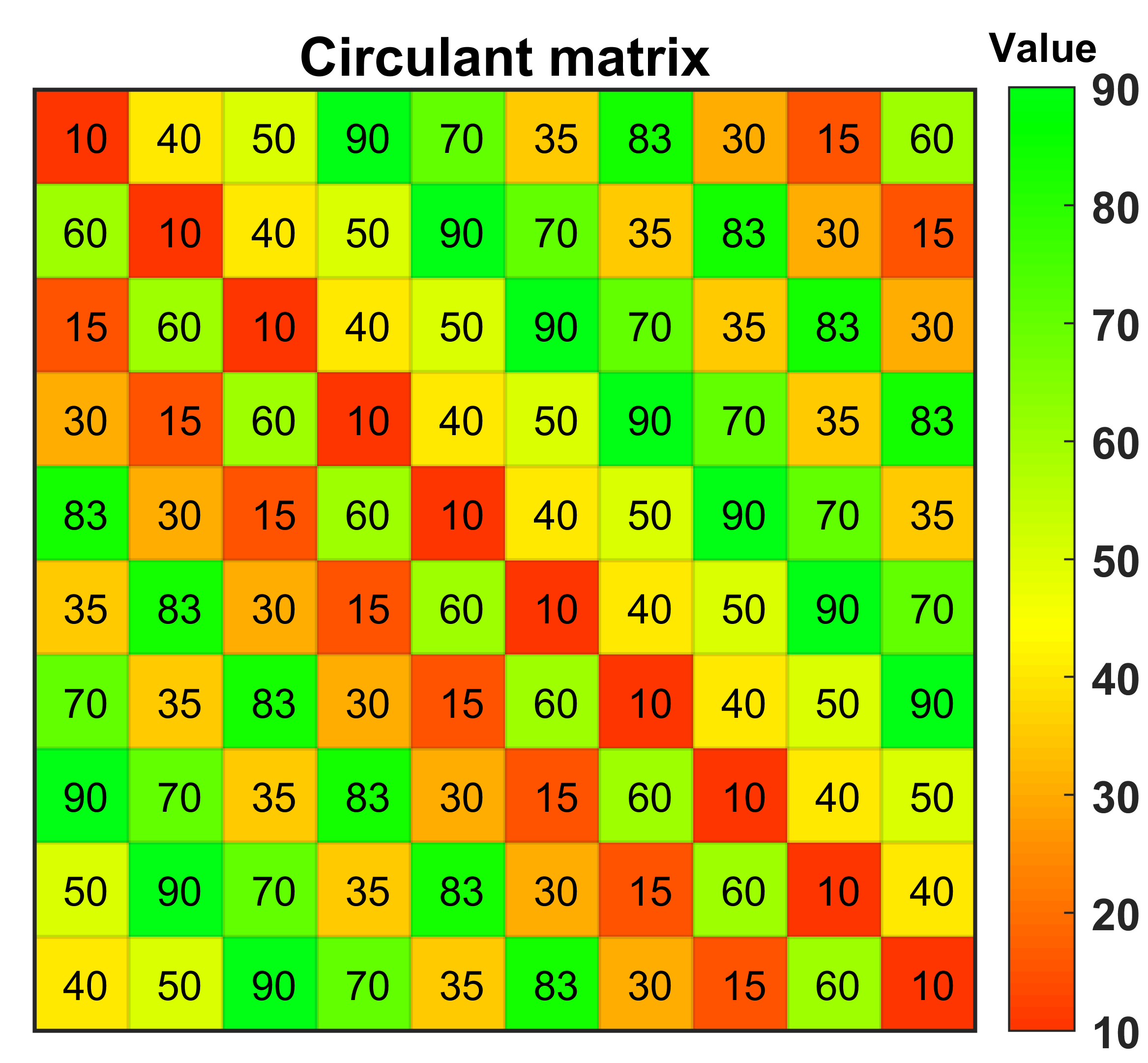}}
\caption{Graphical representation of a circulant matrix.}	 
\label{Fig_Circ-0}
\end{figure}

    The skew circulant operator matrix $\mathbb{S}_{scirc}=\mbox{scirc}(S_1,\dots,S_n)$ is the $n\times n$ circulant matrix followed by a change in sign to all the elements below the main diagonal. Thus,
	$\mathbb{S}_{scirc}=\begin{bmatrix}
	S_{1} & S_{2}  & S_{3} & \cdots  &   S_{n}\\
	-S_{n}  & S_{1} &  S_{2}&\cdots  & S_{n-1}\\
	-S_{n-1}  & -S_{n} &  S_{1}&\ddots  &  S_{n-2}\\
	\vdots & \vdots &\ddots & \ddots & \vdots \\
	-S_{2}  & -S_{3}  &\cdots& - S_{n} & S_{1}
	\end{bmatrix}$. It is well-known that every skew circulant operator matrix is unitarily equivalent to a circulant operator matrix. Details discussion on circulant, skew-circulant and their properties are given in \cite{Davis}. 
		If $S_i\in \mathcal{B(H)}, i=1,\dots,n$, then the imaginary circulant operator matrix $\mathbb{S}_{circ_i}=\mbox{circ}_i(S_1,\dots,S_n)$ is the $n\times n$ matrix whose first row has entries $S_1, \dots, S_n$ and the other rows are obtained by successive cyclic permutations of $i$-multiplies of these entries, i.e.,
	$\mathbb{S}_{circ_i}=\begin{bmatrix}
	S_{1} & S_{2}  & S_{3} & \cdots  &   S_{n}\\
	iS_{n}  & S_{1} &  S_{2}&\cdots  & S_{n-1}\\
	iS_{n-1}  & iS_{n} &  S_{1}&\ddots  &  S_{n-2}\\
	\vdots & \vdots &\ddots & \ddots & \vdots \\
	iS_{2}  & iS_{3}  &\cdots& i S_{n} & S_{1}
	\end{bmatrix}$. Every imaginary circulant operator matrix is unitarily equivalent to a circulant operator matrix. 
	The imaginary skew circulant operator matrix $\mathbb{S}_{scirc_i}=\mbox{scirc}_i(S_1,\dots,S_n)$ is the $n\times n$ imaginary circulant followed by a change in sign to all the elements below the main diagonal. Thus,
	$\mathbb{S}_{scirc_i}=\begin{bmatrix}
	S_{1} & S_{2}  & S_{3} & \cdots  &   S_{n}\\
	-iS_{n}  & S_{1} &  S_{2}&\cdots  & S_{n-1}\\
	-iS_{n-1}  & -iS_{n} &  S_{1}&\ddots  &  S_{n-2}\\
	\vdots & \vdots &\ddots & \ddots & \vdots \\
	-iS_{2}  & -iS_{3}  &\cdots& -i S_{n} & S_{1}
	\end{bmatrix}$.
 \\
 An operator $S\in\mathcal{B}(\mathcal{H})$ is said to be positive, and denoted $S\geq 0$, if $\left<Sx,x\right>\geq 0$ for all $x\in\mathcal{H},$ and is called positive definite, denoted $S>0$, if $\left<Sx,x\right>>0$ for all non zero vectors $x\in\mathcal{H}.$
The {\it numerical range} of $S\in\mathcal{B}(\mathcal{H})$ is defined as
 $ W(S)=\{\langle Sx, x \rangle: x\in \mathcal{H}, \|x\|=1
\}$  and the {\it numerical radius} of $S$, denoted by $w(S)$,  is
defined by $ w(S)=\displaystyle\sup\{|z|: z\in W(S) \}.$
It is known that the set $W(S)$ is a convex subset of the complex plane and that the numerical radius $w(\cdot)$ is a norm on $\mathcal{B}(\mathcal{H})$; being equivalent to the
usual operator norm $\|S\|=\displaystyle \sup  \{ \|Sx \|: x\in \mathcal{H}, \|x\|=1 \}.$ In fact, for
every $S \in \mathcal{B(H)}$, 
\begin{align}\label{p4100}
\frac{1}{2}\|S\|\leq w(S)\leq \|S\|.
\end{align}

The inequalities in \eqref{p4100} are sharp. If $S^2=0$, then the first inequality becomes an equality, on the other hand, the second inequality becomes an equality if $S$ is normal. In fact, for a nilpotent operator $S$  with $S^n=0$, Haagerup and Harpe \cite{HPH} showed that $w(S)\leq \|S\|\cos(\pi/(n+1))$. In particular, when $n=2$, we get the reverse inequality of the first inequality in \eqref{p4100}. The numerical radius has some significant properties, such as the power inequality:
\begin{align}\label{powerineq}
    w(S^n)\leq w^n(S)~~~\mbox{for}~~n=1,2,\dots.
\end{align}
For basic information about numerical radius, one can refer \cite{GRO} and for recent results in this direction, one can see \cite{Bhunia_Sahoo_2025, Daptari_Kittaneh_Sahoo_2024, Sababheh_Moradi_Sahoo_2024, Sahoo_Moradi_Sababheh_Acta_2024, NSD, SND1}.
In a similar way, the $q$-numerical range is defined by  
\begin{align}\label{Eq1.3}
   W_q(S)=\{\langle Sx, y \rangle: x, y\in \mathcal{H}, \|x\|=\|y\|=1, \langle x, y\rangle=q
\}, 
\end{align}
while  the {\it $q$-numerical radius} is
defined by 
\begin{align}\label{EQ1.4}
    w_q(S)=\displaystyle\sup\{|z|: z\in W_q(S) \}.
\end{align}

One can observe that the concept of the $q$-numerical radius is a generalization of the classical numerical radius, for $|q|=1$. This observation follows from the fact that equality would have to hold in the Cauchy-Schwarz inequality $|q|=|\langle x, y\rangle|\leq \|x\|\|y\|=1$, provided $|q|=1$. It is clear that $y=\lambda x$ would need to be true for some $\lambda \in \mathbb{C}, |\lambda|=1$, hence $|\langle Sx, y \rangle|=|\langle Sx, x \rangle|$. 

In 1977, Marcus and Andresen \cite{Marcus} introduced the $q$-numerical range on $n$-dimensional unitary space with an inner product. They showed that the set generated by rotating $q$-numerical range about the origin is an annulus for $q\in \mathbb{C}$ with $q\leq 1$. They also gave the inner and outer radii of this annulus for hermitian operators. In 1984, Nam-Kiu Tsing \cite{Tsing}, established the convexity of the $q$-numerical range. In 1994, C. K. Li et al. \cite{LiCK} described some elementary properties for the $q$-numerical range on finite-dimensional spaces, see also \cite{LiCKNakazato}. In 2002, M.T. Chien and H. Nakazato \cite{ChienMT} described the boundary of the $q$-numerical range of a square matrix using its Davis-Wielandt shell. In 2005, R. Rajic \cite{RajiR} considered a generalization of the $q$-numerical range. In 2007, M.T. Chien and H. Nakazato \cite{ChienMTNakazatoH} again studied the $q$-numerical radius of weighted unilateral and bilateral shift operators and computed the $q$-numerical radius of shift operators with periodic weights. In 2012, M. T. Chien, \cite{ChienMT2012} investigated the $q$-numerical radius of a weighted shift operator when its weights are in geometric sequence and periodic sequence. Recently, the authors of \cite{MMJ, Patra_Roy_2024_OAM, Kaadoud_Moulaharabbi_2024_OAM} established some upper and lower bounds for the $q$-numerical radius of bounded linear operators which generalize some classical numerical radius inequalities. Also, they have presented some useful examples to compare the sharpness of their inequalities for different values of $q\in (0, 1)$. Very recently, Stankovic et al. \cite{HSMKID} investigated some properties for the $q$-numerical radius and presented an improved version of some earlier results from \cite{MMJ}. Also, they established several $q$-numerical radius inequalities for operator matrices defined on the direct sum of Hilbert spaces. Motivated by the work of several authors we have presented $q$-numerical radius of some special type of operator matrices i. e. tridiagonal and circulant operator matrices as defined in the introduction.
 \par 

The present paper consists of a blend of a few distinct bodies: elementary observations and examples, new results, hints to other predictable applications and some original facts.
	
	In this aspect, the rest of the paper is arranged as follows. We have collected some preliminary results in Section $2$ that will be needed to prove our results. We establish certain $q$-numerical radius of $n\times n$ tri-diagonal and anti-tridiagonal operator matrices in Section $3$. In Section $4$, we present $q$-numerical radius inequalities for circulant, skew circulant, imaginary circulant, and imaginary skew circulant operator matrices. 
 Some special cases of our results have been given in this section. Finally, we end up with a conclusion section, which may spark new problems for future research interest.

	\section{ Preliminaries}
In order to make progress our research work, we need the following lemmas to prove our results. The very first lemma represents a consequence of the results obtained in \cite[Proposition 3.1]{Gau}, as well as \eqref{EQ1.4}. 
\begin{lemma}\label{Lem0000}
  Let $T, S\in {\mathcal{B}}(\mathcal{H}), q\in \bar{\mathbb{D}}$ and $\lambda \in \mathbb{C}$. Then we have following properties:
  \begin{enumerate}
      \item [\textnormal{(i)}] $w_q(\lambda T)=|\lambda|w_q(T)$.
       \item [\textnormal{(ii)}] $w_q(T+S)\leq w_q(T)+w_q(S)$.
        \item [\textnormal{(iii)}] $w_q(U^*TU) = w_q(T)$, where $U \in \mathcal{B(H)}$ is an unitary operator.
          \item [\textnormal{(iv)}] $w_{\lambda q}(T)= w_q(T)$ for all $\lambda\in \mathbb{C}$ with $|\lambda|=1$.
  \end{enumerate}
\end{lemma}
Motivated by the work of Hirzallah et al. \cite[Lemma 2.1]{TY}, Stankovic et al. \cite{HSMKID} extended the same result to $q$-numerical radius setting is mentioned below. Same result for semi-Hilbert setting one can see \cite{PINTU, Nirmal2}.
\begin{lemma}\label{lem0001}\textnormal{\cite[Lemma 5.2]{HSMKID}} 
		Let $T, S\in {\mathcal{B}}(\mathcal{H}), q\in \bar{\mathbb{D}}$ and $\theta \in \mathbb{R}$. Then 
		\begin{enumerate}
			\item [\textnormal{(i)}] $w_q\left(\begin{bmatrix}
			O & T\\
			S & O
			\end{bmatrix}\right)=w_q\left(\begin{bmatrix}
			O & S\\
			T & O
			\end{bmatrix}\right).$\\
			\item [\textnormal{(ii)}] $w_q\left(\begin{bmatrix}
			O & T\\
			e^{i\theta}S & O
			\end{bmatrix}\right)=w_q\left(\begin{bmatrix}
			O & T\\
			S & O
			\end{bmatrix}\right)$ for~any~$\theta\in\mathbb{R}$.\\
			\item [\textnormal{(iii)}]  $w_q\left(\begin{bmatrix}
			T & O\\
			O & S
			\end{bmatrix}\right)=w_q\left(\begin{bmatrix}
			S & O\\
			O & T
			\end{bmatrix}\right).$
		\end{enumerate}
	\end{lemma}
	
	The following result is from \cite[Theorem 1.5]{HSMKID}.
	\begin{lemma}\label{Lemma1.3}\textnormal{\cite[Theorem 1.5]{HSMKID}}
		Let $(\mathcal{H}_n)_{n\in \mathbf{N}}$ be a sequence of Hilbert spaces and let $ T_n \in \mathcal{B}(\mathcal{H}_n)$ for all $n\in \mathbf{N}$. If $q\in \Bar{\mathbb{D}}\setminus \{0\}$. Then 
		$$\sup_{n\in \mathbf{N}}w_q(T_n)\leq w_q \left(\bigoplus_{n=1}^{+\infty} T_n \right) \leq \frac{|q|+2\sqrt{1-|q|^2}}{|q|}\sup_{n\in \mathbf{N}}w_q(T_n) .$$
	\end{lemma}
 As a special case of the above result we have the following lemma that represents a generalization of  Hirzallah {\it et al.} \cite[Lemma 2.1 (a)]{TY}.
 \begin{lemma}\label{Lemma1.4}\textnormal{\cite[Corollary 5.1]{HSMKID}}
 Let $T, S \in \mathcal{B}(\mathcal{H})$ and $q\in (0, 1]$. Then
      \begin{align*}
            \max\{w_q(T), w_q(S)\}\leq w_q\left(\begin{bmatrix}
			T& 0\\
			0 & S
			\end{bmatrix}\right)\leq \frac{q+2\sqrt{1-q^2}}{q}\max\{w_q(T), w_q(S)\}.
        \end{align*}
 \end{lemma}
 The following lemma follows from \cite{Nakazato} that will be needed in our purpose. 
\begin{lemma}\label{lem1}\textnormal{\cite{Nakazato}} 
    Suppose $0\leq q\leq 1$ and $T\in M_2(\mathbb{C})$. Then $T$ is unitarily similar to $e^{it}\begin{pmatrix}
        \gamma &a\\
        b & \gamma
    \end{pmatrix}$ for some $0\leq t \leq 2\pi $ and $0\leq b\leq a$. Also
    \begin{align*}
        W_q(T)=e^{it}\left\{\gamma q+r\left((c+pd)\cos(s)+i(d+pc)\sin (s)\right): 0\leq r\leq 1, 0\leq s\leq 2\pi \right\},
    \end{align*}
   with $c=\frac{a+b}{2}, d=\frac{a-b}{2}~\mbox{and}~ p=\sqrt{1-q^2}. $ 
\end{lemma}
\section{$q$-numerical radius inequalities for $n\times n$ tridiagonal and anti-tridiagonal operator matrices}
In this section, inspired by the works of Bani-Domi et al. \cite{BaniKitShat}, Kittaneh et al. \cite{KITSAT} and Sahoo et al. \cite{Sahoo_Moradi_Sababheh_Acta_2024},  our key goal is to investigate the various properties of the $q$-numerical radius with the purpose of obtaining stronger results. We mainly discuss $q$-numerical radius inequalities for special $n\times n$ tridiagonal, and anti-tridiagonal operator matrices. Note that for classical numerical radius, equality holds, whereas our results for $q$-numerical radius: both upper and lower bounds exists. Now, we are in a position to provide a quick Theorem that represents $q$ numerical radius of $n\times n$ tri-diagonal operator matrix.

\begin{theorem}\label{tridiagonalmatrix}
	Let $T, S \in \mathcal{B}(\mathcal{H})$, $q\in (0, 1]$ and $\mathbb{P}=\begin{bmatrix}
	T &S  & O & \cdots  &  O\\
	S  & T &  S&\cdots  & \vdots\\
	O  & S &  T&\ddots  &  O\\
	\vdots & \vdots &\ddots & \ddots & S \\
	O  & O  &\cdots&  S & T
	\end{bmatrix}$ be an $n\times n$ tridiagonal operator matrix in $ \mathcal{B}(\mathcal{H}^{(n)})$.  Then 
	\begin{align*}
	\max \left\{w_q\left(T+\left(2\cos \frac{k\pi}{n+1}\right)S\right)\right\}_{k=1}^n\leq w_q(\mathbb{P})\leq \frac{q+2\sqrt{1-q^2}}{q}\max\left\{ w_q\left(T+\left(2\cos\frac{k\pi}{n+1}\right)S\right): k= 1,\dots, n \right\}.
	\end{align*}	
\end{theorem}
\begin{proof}
   Let $\mathbb{U}=[U_{ij}]$, where $U_{ij}=\sqrt{\frac{2}{n+1}}\begin{bmatrix}     \sin \left(\frac{ij\pi}{n+1}\right) \end{bmatrix}\bigotimes I, 1\leq i, j\leq n$, where $I$ is the identity operator in $\mathcal{B}(\mathcal{H})$. It can be easily check that $\mathbb{U}$ is a unitary operator in $\mathcal{B}(\mathcal{H}^{(n)})$ and then
$$ \mathbb{UP}\mathbb{U}^*=  \begin{bmatrix}
		T+\left(2\cos \frac{\pi}{n+1}\right)S & O  & \cdots  &   O\\
		O  & T+\left(2\cos \frac{2\pi}{n+1}\right)S & \cdots  & O\\
		\vdots & \vdots &\ddots & \vdots \\
		O  & O  &\cdots&  T+\left(2\cos \frac{n\pi}{n+1}\right)S
		\end{bmatrix}.$$

   Now, using the fact $w_q(\mathbb{P})=w_q(\mathbb{UP}\mathbb{U}^*)$, and applying Lemma \ref{Lemma1.3}, we have
\begin{align*}
	\max \left\{w_q\left(T+\left(2\cos \frac{k\pi}{n+1}\right)S\right)\right\}_{k=1}^n\leq w_q(\mathbb{P})\leq \frac{q+2\sqrt{1-q^2}}{q}\max\left\{ w_q\left(T+\left(2\cos\frac{k\pi}{n+1}\right)S\right): k= 1,\dots, n \right\}.
	\end{align*}
   
\end{proof}

Here are some special cases to the Theorem \ref{tridiagonalmatrix}.
\begin{rem}\label{Rem1111}
    For the case $n=3$, we have 
     \begin{align}\label{Eqrem}
            \max\{w_q(T+\sqrt{2}S), w_q(T), w_q(T-\sqrt{2}S)\}&\leq w_q\left(\begin{bmatrix}
			T& S &O\\
			S & T &S\\
                 O&S &T
			\end{bmatrix}\right)\nonumber\\
   &\leq \frac{q+2\sqrt{1-q^2}}{q} \max\{w_q(T+\sqrt{2}S), w_q(T), w_q(T-\sqrt{2}S)\}.
        \end{align}
\end{rem}

\begin{ex}\label{Ex_0}
    Let us consider $\mathbb{A}=\begin{bmatrix}
       2 & 1&0\\
        1 & 2&1\\
        0&1&2
    \end{bmatrix}$ and $q\in [0, 1]$.
   The lower bound of $w_q(\mathbb{A})$ in Remark \ref {Rem1111} is $(2+\sqrt{2})q$, while the upper bound is $(2+\sqrt{2})(q+2\sqrt{1-q^2})$. 
The boundary of $W_q(\mathbb{A})$ for $q = 0.5$ and comparison of $w_q(\mathbb{A})$ with its upper and lower bounds have been shown in Figure \ref{Fig-0}. It can be noted that the analytical expression of $W_q(\mathbb{A})$ and $w_q(\mathbb{A})$ can not be obtained by applying the methods as mentioned in Lemma \ref{lem1} and \cite[Chapter 8, Theorem 3.5]{Gau}. Therefore, the $W_q(\mathbb{A})$ and $w_q(\mathbb{A})$ have been estimated using the numerical algorithm based on the definition of q-numerical radius (i.e.\eqref{Eq1.3}).
\end{ex}

     \begin{figure}[H]
\centering
  {\includegraphics[scale=.52]{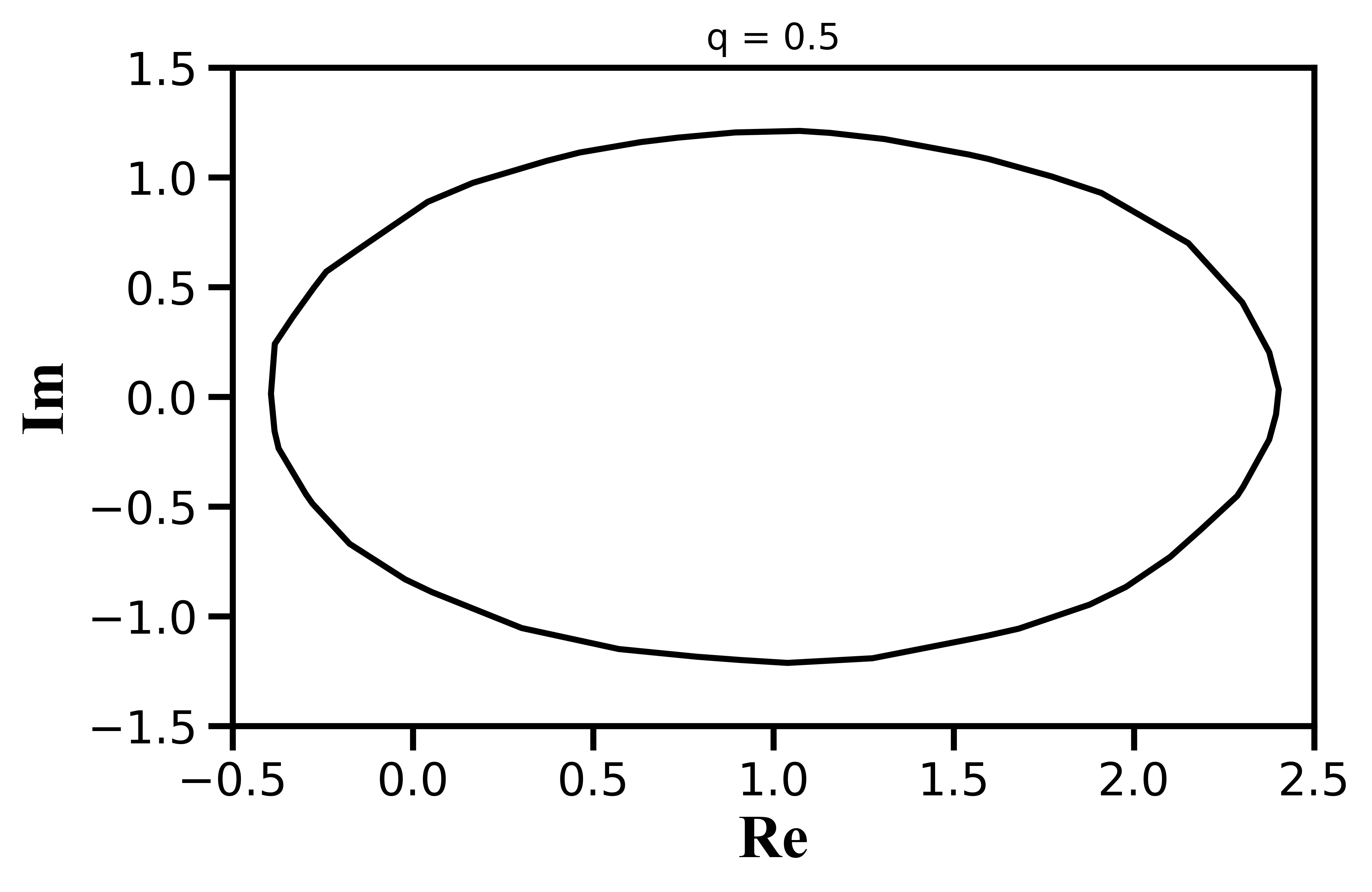}}
{\includegraphics[scale=.52]{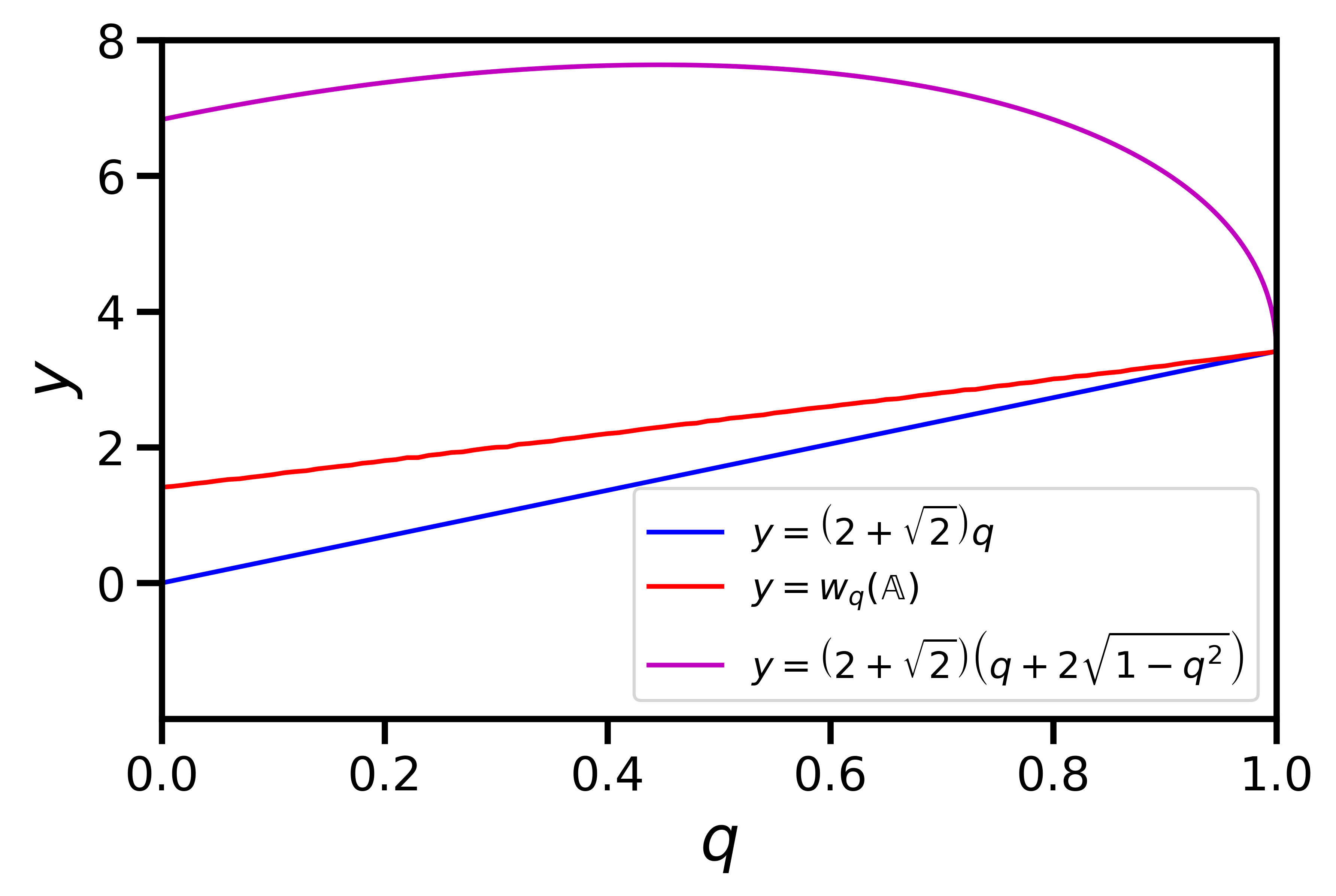}}
\caption{The boundary of the $q$-numerical range of matrix $\mathbb{A}$  for $q = 0.5$ (left column). Comparison of $w_q\left(\mathbb{A}\right) $ with upper and lower bounds \eqref{Eqrem}  (right column). Here  $\mathbb{A}$ is the tridiagonal matrix defined in Example \eqref{Ex_0}.   }	 
\label{Fig-0}
\end{figure}

Here we have established some special cases to the Theorem \ref{tridiagonalmatrix}. 
\begin{rem}\label{Rem_tridiag}
    \begin{enumerate}
        \item [\textnormal{(i)}] If $T=0$, then 
        \begin{align*}
	2\max \left\{\left|\cos \frac{k\pi}{n+1}\right|w_q(S)\right\}_{k=1}^n\leq w_q(\mathbb{P})\leq \frac{q+2\sqrt{1-q^2}}{q}2\max \left\{\left|\cos \frac{k\pi}{n+1}\right|w_q(S): k= 1,\dots, n \right\}.
	\end{align*}
 In particular, 
	$w_q(S)\leq w_q\left(\begin{bmatrix}
			O & S\\
			S & O
			\end{bmatrix}\right)\leq \frac{q+2\sqrt{1-q^2}}{q}w_q(S).$ 
  \item [\textnormal{(ii)}] If $S=0$, then 
        \begin{align*}
	w_q(S)\leq w_q(\mathbb{P})\leq \frac{q+2\sqrt{1-q^2}}{q}w_q(S).
	\end{align*}
   \item [\textnormal{(iii)}] If $T=S$, then 
        \begin{align*}
	\max \left\{\left|1+2\cos \frac{k\pi}{n+1}\right|w_q(T)\right\}_{k=1}^n\leq w_q(\mathbb{P})\leq \frac{q+2\sqrt{1-q^2}}{q}\max \left\{\left|1+2\cos \frac{k\pi}{n+1}\right|w_q(T): k= 1,\dots, n \right\}.
	\end{align*}
  \item [\textnormal{(iv)}] If $S=iT$, then 
        \begin{align*}
	\max \left\{w_q\left(\left(1+\left(2\cos \frac{k\pi}{n+1}\right)i\right)T\right)\right\}_{k=1}^n&\leq w_q(\mathbb{P})\\
 &\leq \frac{q+2\sqrt{1-q^2}}{q}\max \left\{w_q\left(\left(1+\left(2\cos \frac{k\pi}{n+1}\right)i\right)T\right)\right\}_{k=1}^n.
	\end{align*}
    \end{enumerate}
\end{rem}

\begin{rem}\label{Rem3.4}
As a special case of our result, we have the following results which are recently established by Stankovic {\it et al.} \cite{HSMKID}.
    \begin{enumerate}
        \item [\textnormal{(i)}] For the case $n=2$, we have  \begin{align*}
            \max\{w_q(T+S), w_q(T-S)\}\leq w_q\left(\begin{bmatrix}
			T& S\\
			S & T
			\end{bmatrix}\right)\leq \frac{q+2\sqrt{1-q^2}}{q}\max\{w_q(T+S), w_q(T-S)\}.
        \end{align*}
         In particular, 
	$w_q(S)\leq w_q\left(\begin{bmatrix}
			O & S\\
			S & O
			\end{bmatrix}\right)\leq \frac{q+2\sqrt{1-q^2}}{q}w_q(S).$ 
    \item [\textnormal{(ii)}] By setting $q\rightarrow 1$ in Remark \ref{Rem3.4} (i), we have the following equality which is already established by Hirzallah {\it et al.} \cite{TY}.
    $w\left(\begin{bmatrix}
			T& S\\
			S & T
			\end{bmatrix}\right) = \max\{w(T+S), w(T-S)\}.$
    \end{enumerate}

\end{rem}

\begin{rem}
	By setting $q=1$ in Theorem  \ref{tridiagonalmatrix}, the inequality becomes equality for the usual numerical radius, proved by Bani-Domi et al. \cite{BaniKitShat} and for $\mathbb{A}$-numerical radius version of the above  Theorem \ref{tridiagonalmatrix} one can see \cite{KITSAT} and for Hilbert Schmidt version of the Theorem \ref{tridiagonalmatrix} one may look at \cite{Sahoo_Moradi_Sababheh_Acta_2024}.  
\end{rem}

\begin{theorem}\label{tridiagonalmatrix1}
	Let $T, S \in \mathcal{B}(\mathcal{H})$, $q\in (0, 1]$, $\mathbb{P}=\begin{bmatrix}
	T & \alpha^{n-2}S  & 0 & \cdots  &   0\\
	S & T &  \alpha^{n-2}S&\cdots  & \vdots\\
	0  & S &  T &\ddots  &  0\\
	\vdots & \vdots &\ddots & \ddots & \alpha^{n-2}S \\
	0  & 0  &\cdots& S & T
	\end{bmatrix}$ be an $n\times n$ tridiagonal operator matrix, and let $\alpha=e^\frac{\pi i}{2n}$ (the $n$-th root of the imaginary number $i$). Then
	\begin{align*}
	\max\left\{ w_q\left(T+(2\alpha^{n-1}\cos\frac{k\pi}{n+1})S\right)\right\}\leq w_q(\mathbb{P})\leq  \frac{q+2\sqrt{1-q^2}}{q}\max\left\{ w_q\left(T+(2\alpha^{n-1}\cos\frac{k\pi}{n+1})S\right)\right\}, 
	\end{align*}
 ~for~  $k= 1,\dots, n.$ 
\end{theorem}

\begin{proof}

 Let $\mathbb{U}=[U_{ij}]$, where
\begin{align*}
U_{ij} =\left\{ \begin{array}{lll}
\alpha^{J}I & \textnormal{for $ i=j$, $J=0,\dots, n-1$,}\\
O & \textnormal{for $ i \neq j $}. 
\end{array} \right.
\end{align*}
 It can be easily check that $\mathbb{U}$ is a unitary operator in $\mathcal{B}(\mathcal{H}^{(n)})$ and then
$$ \mathbb{U}^*\mathbb{P}\mathbb{U}=  \begin{bmatrix}
		T & \alpha^{n-1}S  & O & \cdots  &  O\\
	\alpha^{n-1}S & T &  \alpha^{n-1}S&\cdots  & \vdots\\
	O  & \alpha^{n-1}S &  T &\ddots  &  O\\
	\vdots & \vdots &\ddots & \ddots & \alpha^{n-1}S \\
	O & O  &\cdots& \alpha^{n-1}S & T
		\end{bmatrix}.$$

   Now, using the fact $w_q(\mathbb{P})=w_q( \mathbb{U}^*\mathbb{P}\mathbb{U})$, and applying Theorem \ref{tridiagonalmatrix}, we have
\begin{align*}
	\max\left\{ w_q\left(T+(2\alpha^{n-1}\cos\frac{k\pi}{n+1})S\right)\right\}\leq w_q(\mathbb{P})\leq  \frac{q+2\sqrt{1-q^2}}{q}\max\left\{ w_q\left(T+(2\alpha^{n-1}\cos\frac{k\pi}{n+1})S\right)\right\}, 
	\end{align*}
 ~for~  $k= 1,\dots, n.$ 
   
\end{proof}

\begin{theorem}\label{tridiagonalmatrix2}
	Let $T, S \in \mathcal{B}(\mathcal{H})$, $q\in (0, 1]$, $\mathbb{P}=\begin{bmatrix}
	T & \omega^{n-1}S  & 0 & \cdots  &   0\\
	\omega S  & T &  \omega^{n-1}S&\cdots  & \vdots\\
	0  & \omega S &  T&\ddots  &  0\\
	\vdots & \vdots &\ddots & \ddots & \omega^{n-1}S \\
	0  & 0  &\cdots& \omega S & T
	\end{bmatrix}$ be an $n\times n$ tridiagonal operator matrix, and let $\omega=e^\frac{2\pi i}{n}$.  Then 
	\begin{align*}
	\max \left\{w_q\left(T+\left(2\cos \frac{k\pi}{n+1}\right)S\right)\right\}\leq w_q(\mathbb{P})\leq \frac{q+2\sqrt{1-q^2}}{q}\max\left\{ w_q\left(T+\left(2\cos\frac{k\pi}{n+1}\right)S\right) \right\},
	\end{align*}	
 $ k= 1,\dots, n$.
\end{theorem}
\begin{proof}

 Let $\mathbb{U}=[U_{ij}]$, where
\begin{align*}
U_{ij} =\left\{ \begin{array}{lll}
\omega^{J}I & \textnormal{for $ i=j$, $J=0,\dots, n-1$,}\\
O & \textnormal{for $ i \neq j $}. 
\end{array} \right.
\end{align*}
 It can be easily check that $\mathbb{U}$ is a unitary operator in $\mathcal{B}(\mathcal{H}^{(n)})$ and then
$$  \mathbb{UP}\mathbb{U}^*=  \begin{bmatrix}
		T & S  & O & \cdots  &  O\\
	S & T &  S&\cdots  & \vdots\\
	O  & S &  T &\ddots  &  O\\
	\vdots & \vdots &\ddots & \ddots & S \\
	O & O  &\cdots& S & T
		\end{bmatrix}.$$

   Now, using the fact $w_q(\mathbb{P})=w_q(  \mathbb{UP}\mathbb{U}^*)$, and implementing Theorem \ref{tridiagonalmatrix}, we have
\begin{align*}
	\max \left\{w_q\left(T+\left(2\cos \frac{k\pi}{n+1}\right)S\right)\right\}\leq w_q(\mathbb{P})\leq \frac{q+2\sqrt{1-q^2}}{q}\max\left\{ w_q\left(T+\left(2\cos\frac{k\pi}{n+1}\right)S\right) \right\},
	\end{align*}	
 $ k= 1,\dots, n$.  
\end{proof}
The following Theorem represents the $q$-numerical radius of $n\times n$ anti-tridiagonal operator matrix. 
\begin{theorem}\label{antitridiagonalmatrix}
	Let $T, S \in \mathcal{B}(\mathcal{H})$, $q\in (0, 1]$, and $\mathbb{P}=\begin{bmatrix}
O &\cdots  &O & S  &   T\\
\vdots & \adots & S&T  &S\\
	O  & \adots &  T&S  &  O\\
	S& \adots &\adots & \adots & \vdots \\
	T  &S  &O& \cdots & O
	\end{bmatrix}$ be an $n\times n$ anti-tridiagonal operator matrix. Then 
	\begin{align*}
	\max\left\{ w_q\left((-1)^{k+1}\left[T+(2\cos\frac{k\pi}{n+1})S\right]\right) \right\}&\leq w_\mathbb{A}(\mathbb{P})\\
 &\leq \frac{q+2\sqrt{1-q^2}}{q}\max\left\{ w_q\left((-1)^{k+1}\left[T+(2\cos\frac{k\pi}{n+1})S\right]\right) \right\},
	\end{align*}	
 $k= 1,\dots, n$.
\end{theorem}
\begin{proof}
    Using the same $\mathbb{U}$ as mentioned in Theorem \ref{tridiagonalmatrix}, we have 
    $$ \mathbb{UP}\mathbb{U}^*=  \begin{bmatrix}
		T+\left(2\cos \frac{\pi}{n+1}\right)S & O  & \cdots  &   O\\
		O  & -(T+\left(2\cos \frac{2\pi}{n+1}\right)S) & \cdots  & O\\
		\vdots & \vdots &\ddots & \vdots \\
		O  & O  &\cdots&  (-1)^{n+1}(T+\left(2\cos \frac{n\pi}{n+1}\right)S)
		\end{bmatrix}.$$

   Now, using the fact $w_q(\mathbb{P})=w_q(\mathbb{UP}\mathbb{U}^*)$, and implementing Lemma \ref{Lemma1.3}, we have our desired result.
\end{proof}
Here are some special cases of the Theorem \ref{antitridiagonalmatrix}. 
\begin{rem}
    \begin{enumerate}
        \item [\textnormal{(i)}] If $T=0$, then 
        \begin{align*}
	2\max \left\{\left|\cos \frac{k\pi}{n+1}\right|w_q(S)\right\}_{k=1}^n\leq w_q(\mathbb{P})\leq \frac{q+2\sqrt{1-q^2}}{q}2\max \left\{\left|\cos \frac{k\pi}{n+1}\right|w_q(S): k= 1,\dots, n \right\}.
	\end{align*}
 In particular, 
	$w_q(S)\leq w_q\left(\begin{bmatrix}
			O & S\\
			S & O
			\end{bmatrix}\right)\leq \frac{q+2\sqrt{1-q^2}}{q}w_q(S).$ 
  \item [\textnormal{(ii)}] If $S=0$, then 
        \begin{align*}
	w_q(S)\leq w_q(\mathbb{P})\leq \frac{q+2\sqrt{1-q^2}}{q}w_q(S).
	\end{align*}
   \item [\textnormal{(iii)}] If $T=S$, then 
        \begin{align*}
	\max \left\{\left|1+2\cos \frac{k\pi}{n+1}\right|w_q(T)\right\}_{k=1}^n\leq w_q(\mathbb{P})\leq \frac{q+2\sqrt{1-q^2}}{q}\max \left\{\left|1+2\cos \frac{k\pi}{n+1}\right|w_q(T): k= 1,\dots, n \right\}.
	\end{align*}
  \item [\textnormal{(iv)}] If $S=iT$, then 
        \begin{align*}
	\max \left\{w_q\left(\left(1+\left(2\cos \frac{k\pi}{n+1}\right)i\right)T\right)\right\}_{k=1}^n&\leq w_q(\mathbb{P})\\
& \leq \frac{q+2\sqrt{1-q^2}}{q}\max \left\{w_q\left(\left(1+\left(2\cos \frac{k\pi}{n+1}\right)i\right)T\right)\right\}_{k=1}^n.
	\end{align*}
    \end{enumerate}
\end{rem}

\begin{rem}
	By setting $q= 1$ in Theorem  \ref{tridiagonalmatrix1}, \ref{tridiagonalmatrix2}, and \ref{antitridiagonalmatrix}, the inequalities becomes equalities for the usual numerical radius proved by Bani-Domi {\it et al.} \cite{BaniKitShat} and for $\mathbb{A}$-numerical radius version of the above  Theorem \ref{tridiagonalmatrix1}, \ref{tridiagonalmatrix2}, and \ref{antitridiagonalmatrix},  established very recently, one can see \cite{KITSAT} and  for Hilbert Schmidt version of the results one may look at \cite{Sahoo_Moradi_Sababheh_Acta_2024}. 
\end{rem}

	\section{$q$-numerical radius inequalities for circulant and skew circulant operator matrices}
	The aim of this section is to investigate certain $q$-numerical radius inequalities for circulant, skew circulant, imaginary circulant, and imaginary skew circulant operator matrices. The very first result is an upper and lower bound for the  $q$-numerical radius of a circulant operator matrix.
	\begin{theorem}\label{Thm1}
		Let $S_i\in \mathcal{B}(\mathcal{H})$ for $1\leq i\leq n$, $q\in (0, 1]$. Then 
		\begin{align*}
		\max\left\{ w_q\left(\sum_{i=1}^n\omega^{k(1-i)} S_i\right) \right\}\leq w_q(\mathbb{S}_{circ})\leq \frac{q+2\sqrt{1-q^2}}{q}\max\left\{ w_q\left(\sum_{i=1}^n\omega^{k(1-i)} S_i\right) \right\},
		\end{align*}
	where $k=0, 1,\dots, n-1$, $\omega=e^{\frac{2\pi i}{n}}$.
	\end{theorem}
	\begin{proof}
		Let $\mathbb{S}_{circ}=\begin{bmatrix}
		S_{1} & S_{2}  & S_{3} & \cdots  &   S_{n}\\
		S_{n}  & S_{1} &  S{2}&\cdots  & S_{n-1}\\
		S_{n-1}  & S_{n} &  S_{1}&\ddots  &  S_{n-2}\\
		\vdots & \vdots &\ddots & \ddots & \vdots \\
		S_{2}  & S_{3}  &\cdots&  S_{n} & S_{1}
		\end{bmatrix},$ 
		let $1,\omega,\omega^2,\dots, \omega^{n-1}$ be $n$ roots of unity with $\omega=e^{\frac{2\pi i}{n}}$
		
		and $\mathbb{U}=\frac{1}{\sqrt{n}}\begin{bmatrix}
		I & I& I  &  \cdots  &   I\\
		I & \omega I& \omega^2I & \cdots  &   \omega^{n-1}I\\
		I & \omega^2I& \omega^4I &  \cdots  &   \omega^{n-2}I\\
		\vdots & \vdots &\vdots &\ddots& \vdots\\
		I & \omega^{n-1}I& \omega^{n-2}I  &  \cdots  &   \omega I\\
		\end{bmatrix}.$
		
		It can be observed that  that $ \bar{\omega}=\omega^{n-1}, \bar{\omega}^2=\omega^{n-2},\cdots, \bar{\omega}^k=\omega^{n-k}$, $k=0, 1,\dots, n-1$, so
		
		$\mathbb{U}^*=\frac{1}{\sqrt{n}}\begin{bmatrix}
		I & I & I  & \cdots  &   I\\
		I & \omega^{n-1}I& \omega^{n-2}I  & \cdots & \omega I\\
		I & \omega^{n-2}I& \omega^{n-4}I  & \cdots  &  \omega^2I\\
		\vdots & \vdots &\vdots & \ddots &\vdots\\
		I&\omega I& \omega^2I  & \cdots  &  \omega^{n-1}I\\
		\end{bmatrix}$ and $\mathbb{UU}^*= \begin{bmatrix}
		I & O  & \cdots  &   O\\
		O  & I & \cdots  & O\\
		\vdots & \vdots &\ddots & \vdots \\
		O  & O  &\cdots&  I
	\end{bmatrix}=\mathbb{U}^*\mathbb{U}.$
		Thus, $\mathbb{U}$ is a unitary operator. Now, 
		\begin{align*}
		\mathbb{US}_{circ}\mathbb{U}^*  
		&=\bigoplus_{k=0}^{n-1}\sum_{i=1}^n \omega^{k(i-1)}
		  S_i\\
		&=\bigoplus_{k=0}^{n-1}\sum_{i=1}^n \bar{\omega}^{k(i-1)}
		  S_i\\
		&=\bigoplus_{k=0}^{n-1}\sum_{i=1}^n \omega^{k(1-i)}
		  S_i.
		\end{align*}
		Using the fact that $w_q(\mathbb{S})=w_q(\mathbb{USU}^*)$ for any $\mathbb{S}\in\mathcal{B}(\mathcal{H}),$ and implementing Lemma \ref{Lemma1.3},
		we get
		\begin{align*}
w_q(\mathbb{S}_{circ})=w_q(\mathbb{US}_{circ}\mathbb{U}^*)&=w_q\left(\bigoplus_{k=0}^{n-1}\sum_{i=1}^n \omega^{k(1-i)}
		  S_i\right)\\
		&=w_q\left(\bigoplus_{k=0}^{n-1}\sum_{i=1}^n \omega^{k(1-i)}
		  S_i\right)\\
		&\leq \frac{q+2\sqrt{1-q^2}}{q}\max\left\{w_q\left(\sum_{i=1}^n  S_i\right),w_q\left(\sum_{i=1}^n\omega^{(1-i)}S_i \right),\dots, w_q\left(\sum_{i=1}^n\omega^{(n-1)(1-i)} S_i\right) \right\}\\
		&=\frac{q+2\sqrt{1-q^2}}{q}\max\left\{ w_q\left(\sum_{i=1}^n\omega^{k(1-i)} S_i\right): k=0, 1,\dots, n-1 \right\}.
		\end{align*}
		Similarly, the left side of the inequality follows from Lemma \ref{Lemma1.3}.
	\end{proof}
	As a special case of Theorem \ref{Thm1}, on the simple scenario when we are dealing with two operators, we have a result that is already established by \cite[Lemma 5.3]{HSMKID}.

 	\begin{corollary}\label{Cor_3.1}
 		Let $T, S\in \mathcal{B}(\mathcal{H})$, $q\in (0, 1]$. Then
 		\begin{align}\label{eqsymmetric}
            \max\{w_q(T+S), w_q(T-S)\}\leq w_q\left(\begin{bmatrix}
			T& S\\
			S & T
			\end{bmatrix}\right)\leq \frac{q+2\sqrt{1-q^2}}{q}\max\{w_q(T+S), w_q(T-S)\}.
        \end{align}
 	\end{corollary}

\begin{ex}\label{Ex_1}
    Let $T=\frac{1}{10} $ and $S=\frac{1}{24}$, so $\mathbb{A}=\begin{bmatrix}
       \frac{1}{10} & \frac{1}{24}\\
         \frac{1}{24} & \frac{1}{10}
    \end{bmatrix}$ and $q\in [0, 1]$, using Lemma \ref{lem1}, we get 
    $$W_q(\mathbb{A})=\left\{\frac{q}{10}+\frac{r}{24}(\cos (s)+i\sqrt{1-q^2}\sin (s)): 0\leq r\leq 1, 0\leq s\leq 2\pi\right\}.$$
    So, $$w_q(\mathbb{A})=\frac{1}{24}+\frac{q}{10}.$$
   The lower bound of $w_q\left(\mathbb{A}\right)$ in \eqref{eqsymmetric} is $ \max\{w_q(T+S), w_q(T-S)\}=  \max\{\frac{17}{120}q, \frac{7}{120}q\}=\frac{17}{120}q$, while the upper bound is $\frac{17}{120}(q+2\sqrt{1-q^2})$.
The boundary of $W_q(\mathbb{A})$ for $q = 0.5$ and comparison of $w_q(\mathbb{A})$ with its upper and lower bounds have been shown in Figure \ref{Fig-1}.

\begin{figure}[H]
\centering
   {\includegraphics[scale=.48]{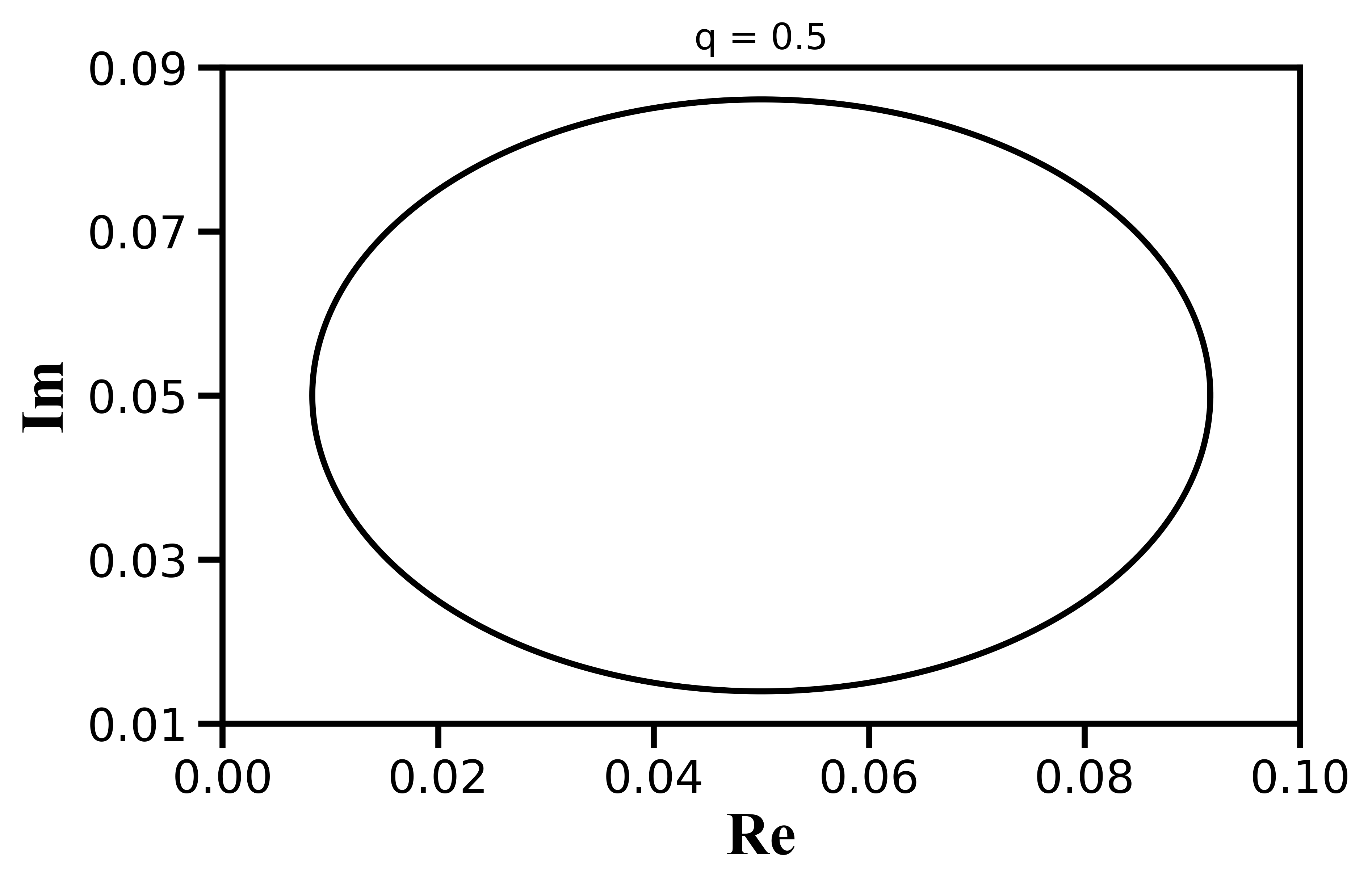}}
 {\includegraphics[scale=.48]{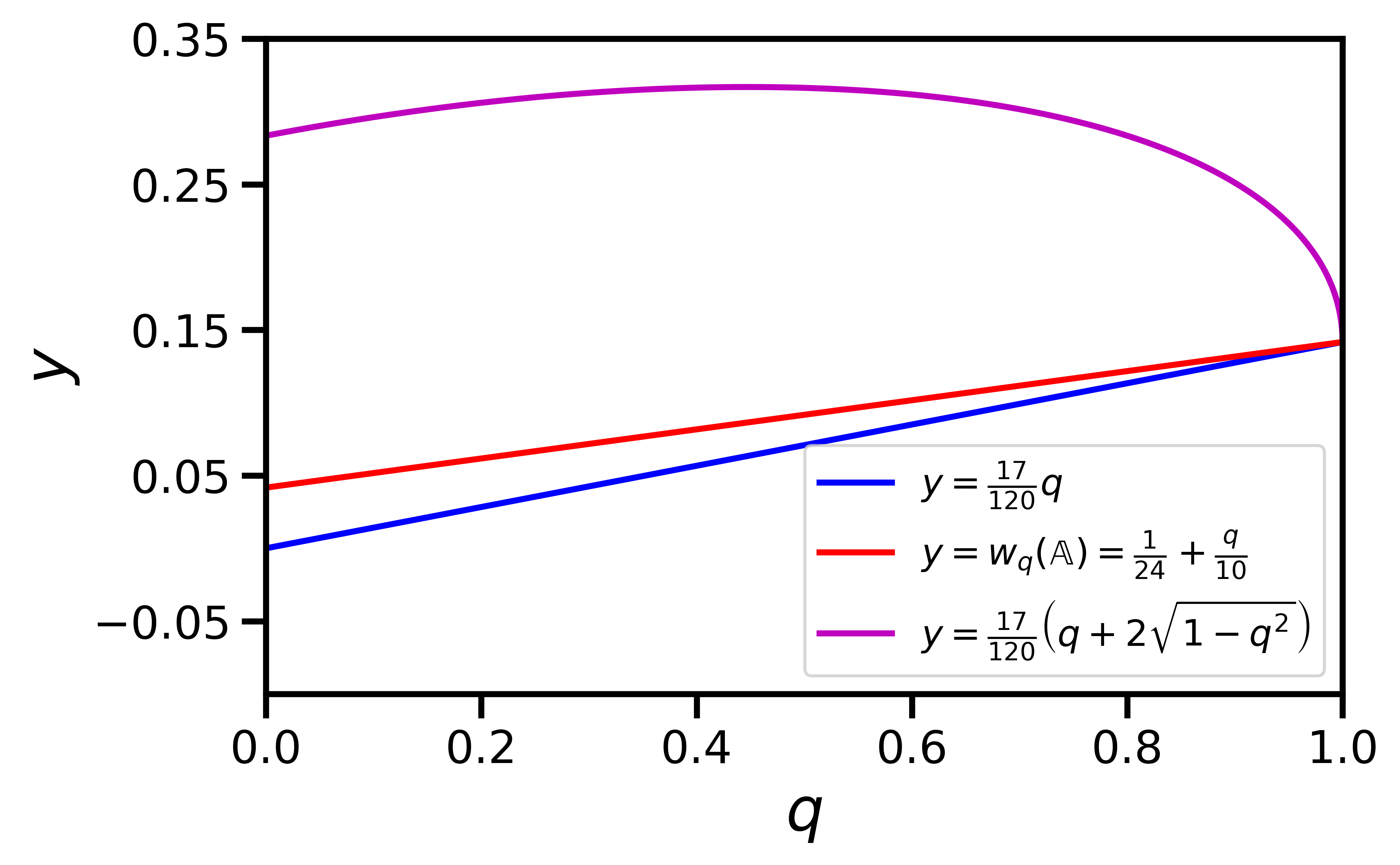}}
\caption{The boundary of the $q$-numerical range of matrix $\mathbb{A}$  for $q = 0.5$ (left column). Comparison of $w_q\left(\mathbb{A}\right) $ with upper and lower bounds \eqref{eqsymmetric}  (right column). Here  $\mathbb{A}=\begin{bmatrix}
       \frac{1}{10} & \frac{1}{24}\\
         \frac{1}{24} & \frac{1}{10}
    \end{bmatrix}$.}	 
\label{Fig-1}
\end{figure}
\end{ex}

  \begin{rem}
      Let $T\in \mathcal{B}(\mathcal{H})$, $q\in (0, 1]$. Then
 		\begin{align}\label{eqsymmetric_1}
          2 w_q(T)\leq w_q\left(\begin{bmatrix}
			T& T\\
			T & T
			\end{bmatrix}\right)\leq \frac{q+2\sqrt{1-q^2}}{q} 2w_q(T).
        \end{align}
  \end{rem}
  \begin{ex}\label{Ex_2}
If $T=I$, $q\in (0, 1]$. Then
	\begin{align}\label{eqsymmetric_2}
          2 w_q(I)\leq w_q\left(\begin{bmatrix}
			I& I\\
			I & I
			\end{bmatrix}\right)\leq \frac{q+2\sqrt{1-q^2}}{q} 2w_q(I).
        \end{align}
       We know from \cite[
Corollary 5.9]{HSMKID} that $w_q\left(\begin{bmatrix}
			I& I\\
			I & I
			\end{bmatrix}\right)=1+q, w_q(I)=q. $
   So from \eqref{eqsymmetric_2} that 
   $$2q\leq 1+q\leq 2(q+2\sqrt{1-q^2}). $$
  \end{ex}
  Let us consider $\mathbb{A}=\begin{bmatrix}
      1& 1\\
			1 & 1
    \end{bmatrix}$, therefore the lower bound of $ w_q\left(\mathbb{A}\right)$ is $2q$ while the upper bound is  $2(q+2\sqrt{1-q^2}).$ 
    
    The boundary of $W_q(\mathbb{A})$ for $q = 0.5$ and comparison of $w_q(\mathbb{A})$ with its upper and lower bounds have been shown in Figure \ref{Fig-2}. The $W_q(\mathbb{A})$ has been estimated using the numerical algorithm based on the definition of $q$-numerical radius (i.e.\eqref{Eq1.3}).

\begin{figure}[H]
\centering
   {\includegraphics[scale=.52]{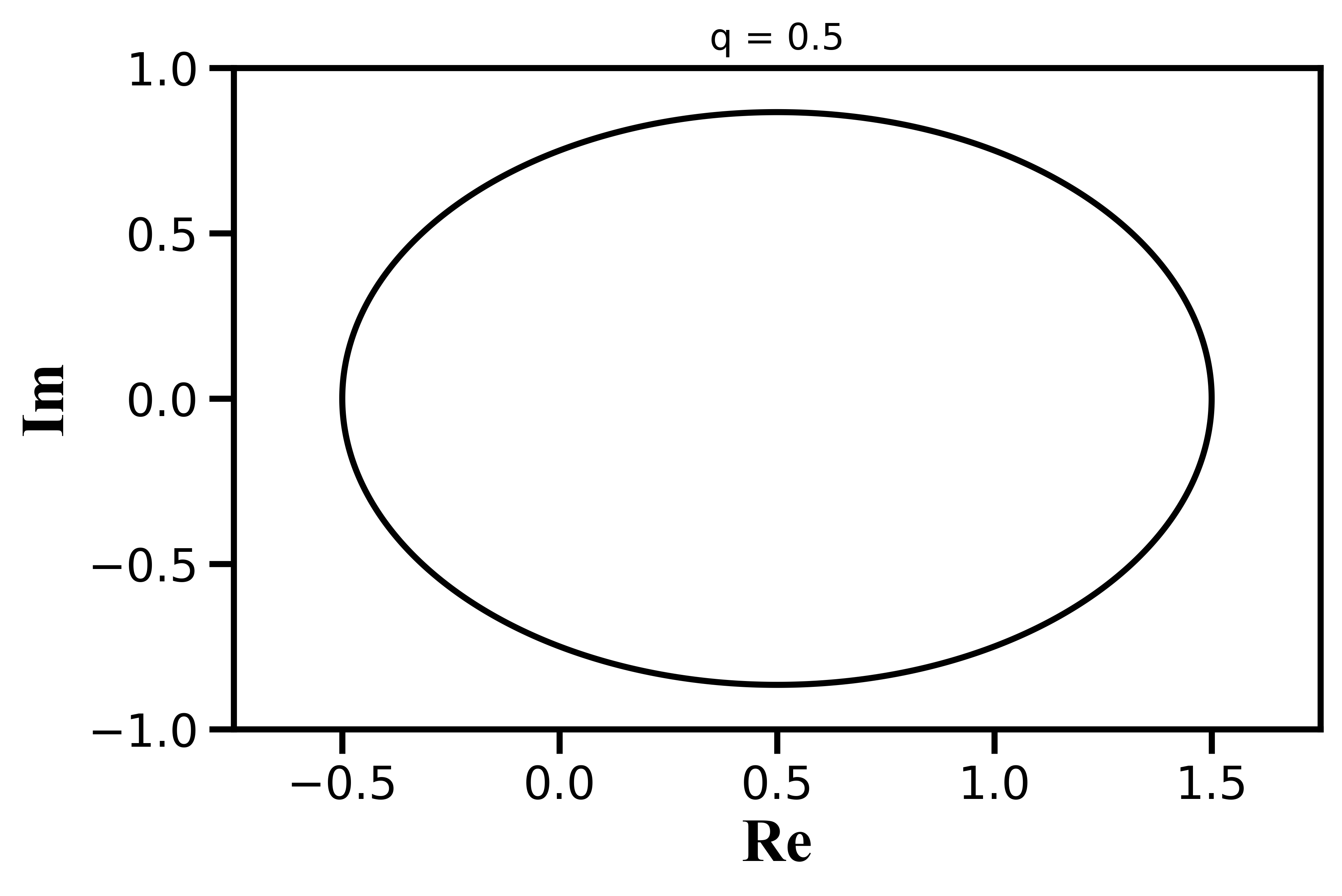}}
  {\includegraphics[scale=.52]{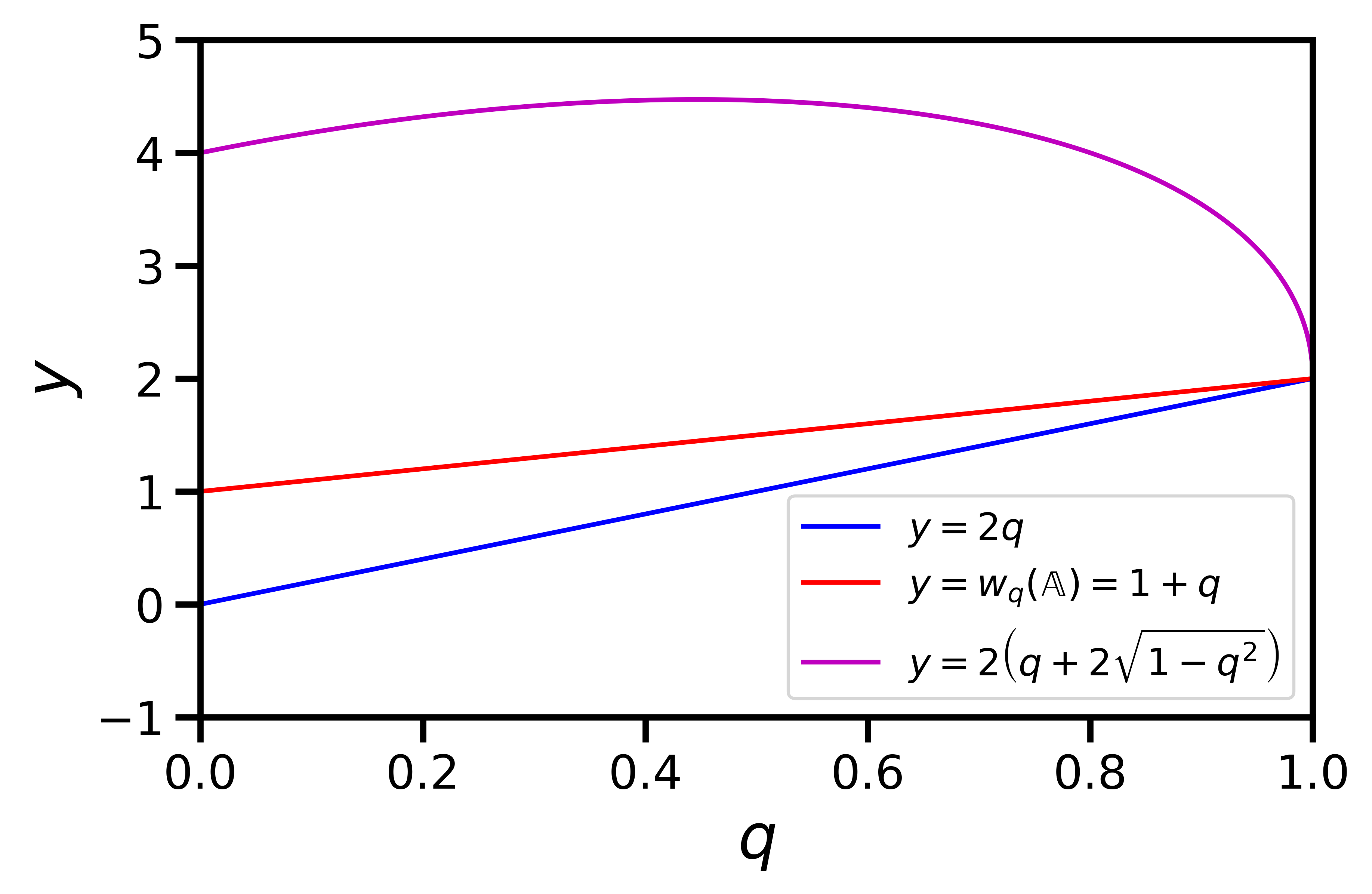}}
\caption{The boundary of the $q$-numerical range of matrix $\mathbb{A}$  for $q = 0.5$ (left column). Comparison of $w_q\left(\mathbb{A}\right) $ with upper and lower bounds \eqref{eqsymmetric_2}  (right column). Here  $\mathbb{A}=\begin{bmatrix}
      1& 1\\
			1 & 1
    \end{bmatrix}$. }	 
\label{Fig-2}
\end{figure}

  \begin{rem}
      For $q= 1$ in Theorem \ref{Thm1}, Corollary \ref{Cor_3.1}, we obtain usual numerical radius equalities proved by Bani-Domi {\it et al.} \cite{BaniKit} and by Hirzallah {\it et al.} \cite[Lemma 2.1 (d)]{TY}. For $\mathbb{A}$-numerical radius version of the above  Theorem \ref{Thm1}, one can see \cite{KITSAT}. 
  \end{rem}

	Our next result is an estimate for $q$-numerical radius of skew circulant operator matrix.
	\begin{theorem}\label{Thm2}
		Let $S_i\in \mathcal{B}(\mathcal{H})$ for $1\leq i\leq n$ and $q\in (0, 1]$. Then
		$$ \max\left\{w_q\bigg(\sum_{i=1}^n(\sigma\omega^k)^{1-i} S_i\bigg)\right\}\leq w_q(\mathbb{S}_{scirc})\leq \frac{q+2\sqrt{1-q^2}}{q}\max\left\{w_q\bigg(\sum_{i=1}^n(\sigma\omega^k)^{1-i} S_i\bigg)\right\},$$  where $ k=0,1,\dots, n-1$, $\sigma=e^{\pi i/n}$ and $\omega =e^{2\pi i/n}$.
	\end{theorem}

	\begin{proof}
		The $n$ roots of the equation $z^n=-1$ are $\sigma, \sigma \omega, \sigma \omega^2, \dots, \sigma \omega^{n-1}$.\\
		Let $\mathbb{S}_{scirc}=\begin{bmatrix}
		S_{1} & S_{2}  & S_{3} & \cdots  &   S_{n}\\
		-S_{n}  & S_{1} &  S_{2}&\cdots  & S_{n-1}\\
		-S_{n-1}  & -S_{n} &  S_{1}&\ddots  &  S_{n-2}\\
		\vdots & \vdots &\ddots & \ddots & \vdots \\
		-S_{2}  & -S_{3}  &\cdots& - S_{n} & S_{1}
		\end{bmatrix}$ 
		and
		\begin{align*}
		\mathbb{U}=\frac{1}{\sqrt{n}}\begin{bmatrix}
		I & \sigma I&\sigma^2 I  & \cdots  &  \sigma^{n-1} I\\
		(\sigma \omega)I & (\sigma \omega)^2 I& (\sigma \omega)^3I  & \cdots  &   (\sigma \omega)^nI\\
		\vdots & \vdots &\vdots & \vdots &\vdots\\
		(\sigma \omega^{n-2})^{n-2}I &  (\sigma \omega^{n-2})^{n-1}I&  (\sigma \omega^{n-2})^nI  & \cdots  &   (\sigma \omega^{n-2})^{2n-1}I\\
		(\sigma \omega^{n-1})^{n-1}I &  (\sigma \omega^{n-1})^{n}I&  (\sigma \omega^{n-1})^{n+1}I  & \cdots  &   (\sigma \omega^{n-1})^{2n-2}I
		\end{bmatrix}.
		\end{align*}
		Using a similar argument as used in the Theorem \ref{Thm1}, we can show that $\mathbb{U}$ is unitary.\\
		Now, 
		\begin{align*}
		\mathbb{US}_{scirc}\mathbb{U}^*=\left(\displaystyle \bigoplus_{k=0}^{n-1}\sum_{i=1}^n(\sigma\omega^k)^{1-i} S_i\right).
		\end{align*}
		Using the property $w_q(\mathbb{S})=w_q(\mathbb{USU}^*)$ for any $\mathbb{S}\in\mathcal{B}(\mathcal{H})$,
		we get
		\begin{align*}
		w_q(\mathbb{S}_{scirc})=w_q(\mathbb{US}_{scirc}\mathbb{U}^*)&=w_q\left(\displaystyle \bigoplus_{k=0}^{n-1}\sum_{i=1}^n(\sigma\omega^k)^{1-i} S_i\right)\\
		&\leq \frac{q+2\sqrt{1-q^2}}{q}\max\left\{w_q\bigg(\sum_{i=1}^n(\sigma\omega^k)^{1-i} S_i\bigg): k=0,1,\dots, n-1\right\}.
		\end{align*}
		Similarly, the left side of the inequality follows from Lemma \ref{Lemma1.3}.
	\end{proof}
	As a special case of the above theorem, we have the following corollary.
	\begin{corollary}\label{Cor_3.2}
		Let $T, S\in \mathcal{B}(\mathcal{H})$ and $q\in (0, 1]$. Then
		\begin{align*}
		\max\{w_q(T+iS),w_q(T-iS)\}\leq w_q\left(\begin{bmatrix}
		T & S\\
		-S & T
		\end{bmatrix}\right)\leq  \frac{q+2\sqrt{1-q^2}}{q}\max\{w_q(T+iS),w_q(T-iS)\}.
		\end{align*}
	\end{corollary}

	\begin{remark}
	    For $q= 1$ in Theorem \ref{Thm2}, we obtain usual numerical radius equalities proved by Bani-Domi {\it et al.} \cite{BaniKit}.   For $\mathbb{A}$-numerical radius version of the above  Theorem \ref{Thm2}, one can see \cite{KITSAT}. 
	\end{remark}

	Theorem \ref{Thm3} provides $q$-numerical radius inequalities for imaginary circulant operator matrices.
	\begin{theorem}\label{Thm3}
		Let $S_i\in \mathcal{B}(\mathcal{H})$ for $1\leq i\leq n$ and $q\in (0, 1]$. Then
		$$\max\left\{w_q\bigg(\sum_{i=1}^n(\alpha\omega^k)^{i-1} S_i\bigg)\right\}\leq  w_q(\mathbb{S}_{circ_i}) \leq \frac{q+2\sqrt{1-q^2}}{q}\max\left\{w_q\bigg(\sum_{i=1}^n(\alpha\omega^k)^{i-1} S_i\bigg) \right\},$$ where $\alpha=e^{\pi i/2n}$,  $k=0,1,\dots, n-1$ and $\omega =e^{2\pi i/n}$.
	\end{theorem}
	\begin{proof}
		
		The $n$ roots of the equation $z^n=i$ are $\alpha, \alpha \omega, \alpha \omega^2, \dots, \alpha \omega^{n-1}$.\\
		Let $\mathbb{S}_{circ_i}=\begin{bmatrix}
		S_{1} & S_{2}  &S_{3} & \cdots  &   S_{n}\\
		iS_{n}  & S_{1} &  S_{2}&\cdots  & S_{n-1}\\
		iS_{n-1}  & iS_{n} &  S_{1}&\ddots  &  S_{n-2}\\
		\vdots & \vdots &\ddots & \ddots & \vdots \\
		iS_{2}  & iS_{3}  &\cdots& i S_{n} & S_{1}
		\end{bmatrix}$ 
		and
		\begin{align*}
		\mathbb{U}=\frac{1}{\sqrt{n}}\begin{bmatrix}
		I &  I& I  & \cdots  &   I\\
		\alpha I & \alpha \omega I& \alpha \omega^2I  & \cdots  &   \alpha\omega^{n-1}I\\
		\alpha^2 I & (\alpha \omega)^2 I& (\alpha \omega^2)^2I  & \cdots  &   (\alpha\omega^{n-1})^2I\\
		\vdots & \vdots &\vdots & \vdots &\vdots\\
		\alpha^{n-1} I & (\alpha \omega)^{n-1} I& (\alpha \omega^2)^{n-1}I  & \cdots  &   (\alpha\omega^{n-1})^{n-1}I
		\end{bmatrix}.
		\end{align*}
		Using a similar argument as used in the Theorem \ref{Thm1}, we can show that $\mathbb{U}$ is unitary.\\
		Now, we have
		\begin{align*}
		\mathbb{U}^{*}\mathbb{S}_{circ_i}\mathbb{U}=\left(\displaystyle \bigoplus_{k=0}^{n-1}\sum_{i=1}^n(\alpha\omega^k)^{i-1} S_i\right).
		\end{align*}
		Using the property $w_q(\mathbb{S})=w_q(\mathbb{U}^*\mathbb{SU})$ for any $\mathbb{S}\in\mathcal{B}(\mathcal{H})$,
		we get
		\begin{align*}
		w_q(\mathbb{S}_{circ_i})=w_q(\mathbb{S}_{circ_i})=w_q(\mathbb{U}^{*}\mathbb{S}_{circ_i}\mathbb{U})&=w_q\left(\displaystyle \bigoplus_{k=0}^{n-1}\sum_{i=1}^n(\alpha\omega^k)^{i-1} S_i\right)\\
		&=w_q\left(\displaystyle \bigoplus_{k=0}^{n-1}\sum_{i=1}^n(\alpha\omega^k)^{i-1} S_i\right)\\
		&\leq \frac{q+2\sqrt{1-q^2}}{q}\max\left\{w_q\bigg(\sum_{i=1}^n(\alpha\omega^k)^{i-1} S_i\bigg): k=0,1,\dots, n-1\right\}.
		\end{align*}
		Similarly, the left side of the inequality follows from Lemma \ref{Lemma1.3}.
	\end{proof}
	\begin{rem}
	      For $q= 1$ in Theorem \ref{Thm3}, we obtain equalities for the usual numerical radius proved by Bani-Domi {\it et al.} \cite{BaniKit} and for $\mathbb{A}$-numerical radius version of the above  Theorem \ref{Thm3}, one can see \cite{KITSAT}. 
	\end{rem}
	As a special case of the above theorem, we have the following corollary.
	\begin{corollary}\label{Cor_3.3}
		Let	$S_1,S_2\in {\mathcal{B}}(\mathcal{H})$ and $q\in (0, 1]$.  Then
		\begin{align*}
	\max\left\{w_q\bigg(S_1+\frac{1+i}{\sqrt{2}}S_2\bigg), w_q\bigg(S_1-\frac{1+i}{\sqrt{2}}S_2\bigg)\right\}&\leq	w_q\left(\begin{bmatrix}
		S_1 & S_2\\
		iS_2 & S_1
		\end{bmatrix}\right)\\
 & \leq \frac{q+2\sqrt{1-q^2}}{q}\max\left\{w_q\bigg(S_1+\frac{1+i}{\sqrt{2}}S_2\bigg), w_q\bigg(S_1-\frac{1+i}{\sqrt{2}}S_2\bigg)\right\}.
		\end{align*}
	\end{corollary}

	In Theorem \ref{Thm4}, we give an estimate for $q$-numerical radius of imaginary skew circulant operator matrices.
	
	\begin{theorem}\label{Thm4}
		Let $S_i\in \mathcal{B}(\mathcal{H})$ for $1\leq i\leq n$ and $q\in (0, 1]$. Then
		$$ \max\left\{w_q\bigg(\sum_{i=1}^n(\beta\omega^k)^{1-i} S_i\bigg)\right\}\leq w_q(\mathbb{S}_{scirc_i})\leq \frac{q+2\sqrt{1-q^2}}{q}\max\left\{w_q\bigg(\sum_{i=1}^n(\beta\omega^k)^{1-i} S_i\bigg)\right\},$$ where $\beta=e^{\frac{-\pi i}{2n}}$, $k=0,1,\dots, n-1$ and $\omega =e^{2\pi i/n}$.
	\end{theorem}
	\begin{proof}
		
		The $n$ roots of the equation $z^n=-i$ are $\beta, \beta \omega, \beta \omega^2, \dots, \beta \omega^{n-1}$.\\
		Let \begin{align*}
	\mathbb{T}_{scirc_i}=\begin{bmatrix}
	S_{1} & S_{2}  & S_{3} & \cdots  &   S_{n}\\
	-iS_{n}  & S_{1} &  S_{2}&\cdots  & S_{n-1}\\
	-iS_{n-1}  & -iS_{n} &  S_{1}&\ddots  &  S_{n-2}\\
	\vdots & \vdots &\ddots & \ddots & \vdots \\
	-iS_{2}  & -iS_{3}  &\cdots& -i S_{n} & S_{1}
	\end{bmatrix}
		\end{align*}
		and
		\begin{align*}
		\mathbb{U}=\frac{1}{\sqrt{n}}\begin{bmatrix}
		I &  I& I  & \cdots  &   I\\
		\beta I & \beta \omega I& \beta \omega^2I  & \cdots  &   \beta\omega^{n-1}I\\
		\beta^2 I & (\beta \omega)^2 I& (\beta \omega^2)^2I  & \cdots  &   (\beta\omega^{n-1})^2I\\
		\vdots & \vdots &\vdots & \vdots &\vdots\\
		\beta^{n-1} I & (\beta \omega)^{n-1} I& (\beta \omega^2)^{n-1}I  & \cdots  &   (\beta\omega^{n-1})^{n-1}I
		\end{bmatrix}.
		\end{align*}
		The rest of the proof follows using a similar method as used in Theorem \ref{Thm3}.
	\end{proof}
	\begin{remark}
	     For $q= 1$ in Theorem \ref{Thm4}, we obtain equalities for the usual numerical radius proved by Bani-Domi {\it et al.} \cite{BaniKit} and for $\mathbb{A}$-numerical radius version of the above  Theorem \ref{Thm4}, one can see \cite{KITSAT}. 
	\end{remark}
    We finish the paper by providing $q$-numerical radius of $2\times 2$ operator matrix. 
	As a special case of the above theorem, we have the following corollary.
	\begin{corollary}\label{Cor_3.4}
		Let	$S_1,S_2\in {\mathcal{B}}(\mathcal{H})$ and $q\in (0, 1]$.  Then
		\begin{align*}
  \max\left\{w_q\bigg(S_1+\frac{1-i}{\sqrt{2}}S_2\bigg), w_q\bigg(S_1-\frac{1-i}{\sqrt{2}}S_2\bigg)\right\}&\leq 
		w_q\left(\begin{bmatrix}
		S_1 & S_2\\
		-iS_2 & S_1
		\end{bmatrix}\right)\\
  &\leq \frac{q+2\sqrt{1-q^2}}{q}\max\left\{w_q\bigg(S_1+\frac{1-i}{\sqrt{2}}S_2\bigg), w_q\bigg(S_1-\frac{1-i}{\sqrt{2}}S_2\bigg)\right\}.
		\end{align*}
	\end{corollary}

	\section{Conclusion}
   
In this paper, we have presented different types of new $q$-numerical radius inequalities, which depend on the structure of circulant, skew circulant, imaginary circulant, imaginary skew circulant, tridiagonal, and anti-tridiagonal operator matrices.  
	\par

    The exploration of $q$-numerical radius inequalities for 
$n\times n$ circulant, skew circulant, imaginary circulant, and imaginary skew circulant matrices involves a deep understanding of their properties and how the parameter 
$q$ modifies these properties. The $q$-numerical radius can provide important insights into the stability and behavior of these matrices in various applications in linear algebra and operator theory.

Advanced inequalities involving the $q$-numerical radius would mainly depend on the specific structure and properties of the matrix (e. g., circulant, skew circulant) and the values of $q$ varying from $0$ and $1$, and they would be derived by examining the properties of these matrices and their interactions with the chosen inner product.
By employing similar method to different special operator matrices such as left circulant, skew left circulant, left imaginary circulant, and left imaginary skew circulant operator matrices, as defined in \cite{Daptari_Kittaneh_Sahoo_2024}, it is possible to obtain further $q$-numerical radius inequalities. Future research on this subject could lead to the study of a fascinating field for further investigation.

\small{\bf Data availability:}
The authors declare that data sharing is not applicable to this article as no datasets were
generated or analysed during the current study.\\
\noindent
\\
\small{\bf Conflict of interest:} The authors declare that there is no conflict of interest.

	\bibliographystyle{amsplain}
	
\end{document}